\documentclass [ smallextended ] {svjour3}
\usepackage{cite}
\usepackage{amssymb}
\usepackage{amsfonts}
\usepackage{amsmath}
\usepackage{graphicx}
\usepackage{tikz}
\def\E{\frac{1}{T}\int_{\frac{5}{8}T}^{\frac{7}{8}T}E(U(x,t),B_{\lambda't}(0))\,dt}
\def\En{\frac{1}{T_n}\int_{\frac{5}{8}T_n}^{\frac{7}{8}T_n}E(U(x,t),B_{\lambda't}(0))\,dt}
\newtheorem{thm}{Theorem}[section]
\newtheorem{lemm}{Lemma}[section]
\numberwithin{equation}{section}

\begin{document}

\author{J. Nahas}
\institute{J. Nahas
\at 
\'Ecole Polytechnique F\'ed\'erale de Lausanne \\
MA B1 487\\ 
CH-1015 Lausanne
\\ \email{joules.nahas@epfl.ch}
}


\journalname{Calculus of Variations}

\title{Scattering of wave maps from $\mathbb R^{2+1}$ to general targets}

\maketitle

\begin{abstract}
We show that smooth, radially symmetric wave maps $U$ from $\mathbb R^{2+1}$ to a compact target manifold $N$, where $\partial_r U$ and $\partial_t U$ have compact support for any fixed time, scatter. The result will follow from the work of 
Christodoulou and Tahvildar-Zadeh, and Struwe, upon proving that for $\lambda' \in (0,1)$, energy does not 
concentrate in the set $$K_{\frac{5}{8}T,\frac{7}{8}T}^{\lambda'} = \{(x,t) \in \mathbb R^{2+1} \vert \hspace{5pt}  |x| \leq \lambda' t, t \in [(5/8)T,(7/8)T] \}.$$
\end{abstract}


\section{Introduction}
In this work we consider the initial value problem for wave maps from $\mathbb R^{2+1}$ to a compact target manifold $(N,\langle\cdot,\cdot\rangle)$,
\begin{equation}
\left\{
\begin{array}{c l}
  & \partial_{\alpha}\partial^{\alpha}U=B(U)(\partial_{\alpha}U,\partial^{\alpha}U), \notag \\
& U(x,0)=U_0(x), \partial_tU(x,0)=U_1(x), \textrm{ } x \in \mathbb R^2,
\end{array}
\right.
\label{wv-mp-q}
\end{equation}
where $B$ is the second fundamental form of $(N,\langle\cdot,\cdot\rangle) \hookrightarrow \mathbb R^d$.
Much is known about this system; we refer readers to \cite{MR2233925}, \cite{MR1674843}, and references therein.

Concerning radially symmetric wave maps, Christodoulou and Tahvildar-Zadeh in \cite{MR1223662} proved global well-posedness for smooth wave maps
to targets that satisfied certain bounds on the second fundamental form of geodesic spheres, in addition to being either compact or having bounded structure functions. These results were obtained by showing that energy does not concentrate at the origin, along with pointwise estimates on the fundamental solution to the linear problem.

Struwe in \cite{MR1985457} extended this result to radially symmetric wave maps from $\mathbb R^{2+1}$ to spheres $S^k$, and later in \cite{MR1971037} to general targets, by showing with energy estimates and rescaling, that energy cannot concentrate at the origin. Concerning 
asymptotic behavior for
radially symmetric wave maps, Christodoulou and Tahvildar-Zadeh in \cite{MR1230285} proved pointwise estimates 
which imply scattering
for smooth wave maps that differ from a constant map within a compact set to 
targets satisfying the same conditions as in \cite{MR1223662}.

Let
\begin{equation}
\cos(t \sqrt{-\Delta})f(x) + \frac{\sin(t\sqrt{-\Delta})}{\sqrt{-\Delta}}g(x)
\notag
\end{equation}
denote the solution at time $t$ to the linear wave equation
\begin{equation}
\left\{
\begin{array}{c l}
  & \partial_{\alpha}\partial^{\alpha}U=0, \notag \\
& U(x,0)=f(x), \partial_tU(x,0)=g(x), \textrm{ } x \in \mathbb R^2.
\end{array}
\right.
\notag
\end{equation}
We will use similar methods as in \cite{MR1985457}, \cite{MR1990477}, \cite{MR1971037}, and \cite{MR1230285}
to prove our main result.
\begin{thm}
\label{main}
For a smooth, radially symmetric wave map $U(x,t)$ to a compact target manifold $(N,\langle\cdot,\cdot\rangle) \hookrightarrow \mathbb R^d$ that for each $t$ differs from a constant map within a compact set,
 there exists functions $U_{+,0},U_{+,1}:\mathbb R^{2} \rightarrow \mathbb R^d$ such that
\begin{equation}
\lim_{t \rightarrow \infty}\left \|U(x,t)-
\cos(t \sqrt{-\Delta})U_{+,0}(x) - \frac{\sin(t\sqrt{-\Delta})}{\sqrt{-\Delta}}U_{+,1}(x) \right \|_{\dot{H}^1}=0.
\notag
\end{equation}
\end{thm}
 
In Section 2, we review the work done on radially symmetric wave maps, with emphasis on results which we will use to prove Theorem \ref{main}, in Section 3.
We use the following notation. Let
\begin{equation}
K_{S,T}^{\lambda} =\{(x,t) \in \mathbb R^{2+1} \vert \hspace{5pt}  |x| \leq \lambda t, t \in [S,T]\}.
\notag
\end{equation}
Energy will be denoted by
\begin{equation}
E(U(x,t),S) = \int_{S} \langle \partial_{\alpha}U,\partial_{\alpha}U \rangle \,dx.
\notag
\end{equation}
With $r=|x|$, we will denote light cone coordinates as $u=t-r$, $v=t+r$. The statement '$a \lesssim b$' will mean
the quantity $a$ is less than $b$ multiplied by a fixed constant.

\section{A brief review of radially symmetric wave maps} 
We will prove our main result by showing that energy does not concentrate in the 
set $K_{\frac{5}{8}T,\frac{7}{8}T}^{\lambda'}$. Scattering will then follow by the work of Struwe in \cite{MR1971037}, and
Christodoulou and Tahvildar-Zadeh in \cite{MR1230285}. We briefly describe these results here.

In \cite{MR1230285}, the authors prove a series of energy estimates, which are then used
in a bootstrap argument. We mention two in particular that will be used later. For $0 < \lambda' < \lambda'' < 1$ (see page 37 of \cite{MR1230285}),
\begin{equation}
\lim_{t \rightarrow \infty}E(U(x,t),B_{\lambda''t}(0) \setminus B_{\lambda't}(0))=0,
\label{nnr-cn-dc}
\end{equation}
and (see page 39 of \cite{MR1230285})
\begin{equation}
\lim_{T \rightarrow \infty}\frac{1}{T}\int \int_{K_{T/2,T}^{\lambda'}}|U_t|^2 = 0.
\label{kntc-nrgy-dc}
\end{equation}
Their bootstrap argument hinges on the Bondi energy decaying for large $u$,
\begin{equation}
\mathcal{E}(u) \equiv \int_{u}^{\infty}r|\partial_vU|^2\,dv
\rightarrow 0 \textrm{ as }u \rightarrow \infty.
\end{equation} 
In order to control $\mathcal{E}(u)$,
define (see \cite{MR1230285}, page 34)
\begin{equation}
\mathcal{E}_{\lambda'}(u)
\equiv
\int_{((1+\lambda')/(1-\lambda'))u}^{\infty}
2r|\partial_{v}U|^2\,dv,
\end{equation}
which will approach $0$ as $u \rightarrow \infty$, and observe
that for $u=(1-\lambda')t$ (ibid, page 43),
\begin{equation}
\frac{1}{T}\int_{\frac{5}{8}T}^{\frac{7}{8}T}\mathcal{E}(u)\,dt
=\frac{1}{T}\int_{\frac{5}{8}T}^{\frac{7}{8}T}[\mathcal{E}_{\lambda'}(u)
+E(U(x,t),B_{\lambda't}(0))]\,dt.
\notag
\end{equation}
By using assumptions on the second fundamental form of geodesic spheres of $N$, along with energy estimates, the authors show (ibid, page 42)
\begin{equation}
\lim_{T \rightarrow \infty}\frac{1}{T}\int_{\frac{5}{8}T}^{\frac{7}{8}T}E(U(x,t),B_{\lambda't}(0))\,dt=0,
\label{dcy}
\end{equation}
which implies the necessary decay on $\mathcal{E}(u)$.

This is the only place where the bounds on the second fundamental form come into play.
The rest of the paper is a bootstrap argument
that proves the main result, 
\begin{thm}
Let $\mathcal{C}_u^+$ (resp. $\mathcal{C}_u^-$) be the interior of the future (resp. past) light cone
with vertex at $(t=u, r=0)$ in $M=\mathbb R^{2,1}$.
For a smooth, radially symmetric wave map $U$ that satisfies \eqref{dcy}, there holds for $u>0$ and some $c>0$,
\begin{equation}
\textrm{diam}(U(\mathcal{C}_u^+)) \leq \frac{c}{\sqrt{u}}.
\notag
\end{equation} 
\end{thm}
\noindent
along with the two following estimates that we require.
\begin{lemm}
For a smooth, radially symmetric wave map $U$ that satisfies \eqref{dcy}, and has derivatives $\partial_t U$ and $\partial_r U$ at $t=0$ with compact support, 
there exists $u_0 > 0$, and $c>0$ so that for $u > u_0$,
\begin{equation}
|\partial_v U | \leq \frac{c}{v^{\frac{3}{2}}},
|\partial_u U | \leq \frac{c}{v^{\frac{1}{2}}u}.
\notag
\end{equation}
\label{c-t-z-2}
\end{lemm}

\begin{lemm}
Let $U$ be a smooth, radially symmetric wave map  that satisfies \eqref{dcy}, and has derivatives $\partial_t U$ and $\partial_r U$ at $t=0$ with compact support
in a ball of radius $R$ centered at the origin. Let $(t_0,r_0)$ be a fixed point, $r_0 > R$. Then there exists a continuous, increasing function $c(u)$ such that
\begin{align}
\sup_{r \geq r_0}r^{3/2}|\partial_{v}U(u+r,r)| & \leq c(u), \notag \\
\sup_{r \geq r_0}r^{1/2}|\partial_{u}U(u+r,r)| & \leq c(u). \notag
\end{align}
\label{c-t-z-3}
\end{lemm}

We can combine the estimates in Lemma \ref{c-t-z-2} and Lemma \ref{c-t-z-3} to obtain estimates on the derivatives of $U$ for all $r$, and $t \ge 0$. For fixed $u_0>0$, and large enough $r$ with $t-r \leq u_0$, we have 
\begin{equation}
r^2 \pm 2rt+t^2+1 \leq r^2+4r^2+4r^2+r^2,
\notag
\end{equation}
which implies
\begin{equation}
\frac{1}{r} \lesssim \frac{1}{\sqrt{(t \pm r)^2+1}}.
\label{cmb-stm}
\end{equation}
We can use
the estimates in Lemma \ref{c-t-z-2} when $t-r >  u_0$, then use Lemma \ref{c-t-z-3} combined with \eqref{cmb-stm} when $t-r \leq  u_0$  and $r \ge r_0$. The only region when $t \ge 0$ that is not covered in this dichotomy is bounded in time, which can be handled with the local existence theory. With this, we have the following theorem. 
\begin{thm}
For a smooth, radially symmetric wave map $U$ that satisfies \eqref{dcy}, and has derivatives $\partial_t U$ and $\partial_r U$ at $t=0$ with compact support, there is a $c$ such that for $t \ge 0$,
\begin{equation}
|\partial_v U | \leq \frac{c}{(v^2+1)^{\frac{3}{4}}},\textrm{ and }
|\partial_u U | \leq \frac{c}{(v^2+1)^{\frac{1}{4}}\sqrt{u^2+1}}.
\notag
\end{equation}
\label{c-t-z}
\end{thm}

In \cite{MR1971037}, it is shown that
energy does not concentrate at the origin at some time $T$, since this is the only obstacle to 
global well-posedness by \cite{MR1223662}. By finite speed of propagation, we may assume $\partial_t U$ and $\partial_r U$ have compact support. Arguing by contradiction, for $\varepsilon_1$ small enough, one finds a radius $R(t)$ such that
\begin{equation}
\varepsilon_1  \leq E(U(x,t),B_{6R(t)}(0)) \leq 2 \varepsilon_1 < \liminf_{t \rightarrow T}E(U(x,t),B_{T-t}(0)),
\notag
\end{equation}
from which it follows that for $|\tau| \leq 5R(t)$,
\begin{equation}
 E(U(x,t+\tau),B_{R(t)}(0)) \leq 2\varepsilon_1, 
\label{sml-nrgy-trp}
\end{equation}
and
\begin{equation}
\varepsilon_1  \leq E(U(x,t+\tau),B_{11R(t)}(0)). 
\label{nrgy-trp-1}
\end{equation}
It can also be shown that 
\begin{equation}
\lim_{t \rightarrow T}R(t)/(T-t)=0.
\label{r-grth-1}
\end{equation}

Using estimates on the kinetic energy, one can find a sequence of intervals $\{(t_l-R(t_l),t_l+R(t_l))\}$ with
$t_l \rightarrow T$ so that
\begin{equation}
\lim_{l \rightarrow \infty }\frac{1}{R(t_l)}\int_{(t_l-R(t_l),t_l+R(t_l))}\left ( \int _{B_{T-t}(0)}|U_t|^2\,dx\right )\,dt = 0.
\notag
\end{equation}
By rescaling with 
$U_l(t,x) \equiv U(t_l+R(t_l) t,R(t_l) x)$,
one obtains a sequence of wave maps $\{U_l\}$
with
\begin{equation}
\lim_{l \rightarrow \infty }\int_{(-1,1)}\left ( \int _{D_l(t)}|\partial_tU_l|^2\,dx\right )\,dt = 0,
\label{rscld-k-dc-1}
\end{equation}
where $D_l(t)=\{x \vert \hspace{5pt} R(t_l)|x|\leq t_l+R(t_l)(T-t) \}$.

With these estimates, it can be shown that $U_l$ converges to a harmonic map $\overline{U}$. Specifically, \eqref{rscld-k-dc-1} shows that $\overline{U}$ satisfies a harmonic map equation, 
that $U$ has finite energy implies that $\overline{U}$ must have finite energy,
 and
\eqref{r-grth-1} shows that $\overline{U}$ is a map from all of $\mathbb R^2$
to $N$.
Since $N$ is compact, $\overline{U}$ must be constant. With \eqref{sml-nrgy-trp} and some geometric estimates, one can then show that locally, the energy of $U_l$ tends to $0$ as $l \rightarrow \infty$,
contradicting the lower bound in \eqref{nrgy-trp-1}. In particular, Struwe proved the following result in \cite{MR1971037}.
\begin{thm}
Let $\{U_l\}$ be a sequence of radially symmetric wave maps from $\mathbb R^{2+1}$ to a compact manifold $N$ with total energy uniformly bounded. Let $D_l(t)$ be
a family of subsets of $\mathbb R^2$ that obeys $\limsup_{l \rightarrow \infty}D_l(t)=\mathbb R^2$ and, together with $\{U_l\}$, satisfies \eqref{rscld-k-dc-1}.
Then for $\varepsilon$ small enough, and if for $t \in [-1,1]$, 
\begin{equation}
E(U_l(x,t),B_{1}(0))
< \varepsilon,
\notag
\end{equation}
the energy of $U_l$ on any compact set approaches $0$ as $l \rightarrow \infty$. 
\label{s}
\end{thm}

\section{Proof of main result}
With the results from the previous section,
we prove Theorem \ref{main}.
Using Theorem \ref{s}, we will show \eqref{dcy}, then use this fact to apply Theorem \ref{c-t-z}.

We argue the decay of energy by contradiction.
Suppose that for all $\lambda'\in (0,1)$, it is not true that
\[\lim_{T \rightarrow \infty}\E = 0.\]
Since energy is positive and bounded, there is some $\lambda'\in (0,1)$ so that
\begin{equation}
\limsup_{T \rightarrow \infty}\E = \eta,
\label{dfn-lmb}
\end{equation}
where $0 < \eta < \infty$.
Pick $\{T_n\}_{n \in \mathbb N}$ such that $T_n \rightarrow \infty$, and 
\begin{equation}
\lim_{n \rightarrow \infty}\En = \eta.
\end{equation}

In order to produce the sequence $U_l$ in Theorem \ref{s}, we require the lower bound in \eqref{nrgy-trp-1}, which we now prove. By energy conservation (for $t < t'$ and $R>0$), 
\begin{equation}
 E(U(x,t),B_{R}(0)) \leq E(U(x,t'),B_{R+(t'-t)}(0)),
\label{nrgy-cnsrv}
\end{equation}
any energy that enters or leaves
$K^{\lambda'}_{\frac{5}{8}T,\frac{7}{8}T}$ must pass through the surrounding region.
By \eqref{nnr-cn-dc}, energy just outside $K^{\lambda'}_{\frac{5}{8}T,\frac{7}{8}T}$ must decay with time. This keeps energy from rapidly fluxuating in $K^{\lambda'}_{\frac{5}{8}T,\frac{7}{8}T}$, so after sufficient time, the energy at a fixed time in $K^{\lambda'}_{\frac{5}{8}T,\frac{7}{8}T}$ must stay away from $0$. This
argument is formalized in the following lemma.

\begin{lemm}
 Fix $\lambda'' \in (\lambda',1)$. There is an $\alpha = \alpha(\lambda',\lambda'')\in (0,1)$ such that for large enough $n$,
$t \in [\frac{5}{8}T_n,\frac{7}{8}T_n]$, it follows that 
$$\alpha \eta < E(U(x,t),B_{\lambda't}(0)).$$
\label{lwr-nrg-bnd}
\end{lemm}
\begin{proof}
From \eqref{nnr-cn-dc}, we can pick $n$ big enough so that 
$$E(U(x,t),B_{\lambda''t}(0) \setminus B_{\lambda't}(0)) < \beta \eta,$$ for $\beta=\beta(\lambda',\lambda'')$ to be chosen later and $t \ge \frac{5}{8}\frac{1-\lambda''}{1-\lambda'}T_n$. For perhaps even larger $n$, 
we can have that 
\begin{equation}
\left |\En - \eta \right | < \gamma \eta, \label{contradict-me}
\end{equation}
for $\gamma=\gamma(\lambda',\lambda'') \in (0,1)$ which we will specify below.

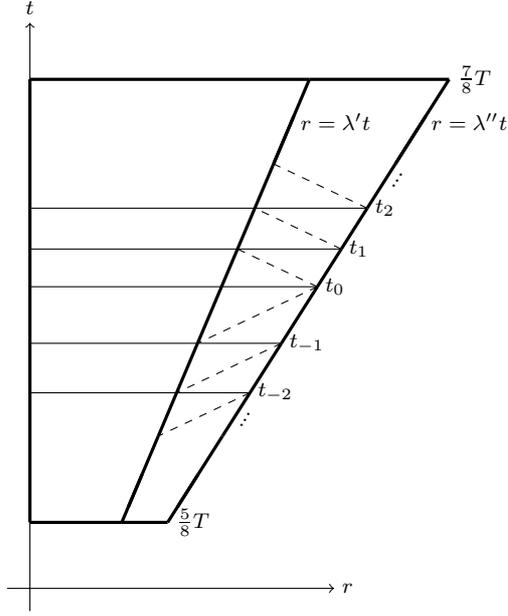
\begin{figure}
    \beginpgfgraphicnamed{Fig1}
\begin{center}
\begin{tikzpicture}[domain=0:4]

\draw[->] (-.3,2) -- (4,2) node[right] {$r$};

\draw[->] (0,1.7) -- (0,9.5) node[above] {$t$};


\draw[very thick] (0,2.875) --(1.2075,2.875); 
 
\draw[very thick] (1.2075,2.875) --(1.81125,2.875)node[right] {$\frac{5}{8} T$}; 
 
\draw[very thick] (0,8.75) --(3.675,8.75); 
 
\draw[very thick] (3.675,8.75) --(5.5125,8.75)node[right] {$\frac{7}{8} T$}; 
 
\draw[very thick] (0,2.875) --(0,8.75); 
 
 
\draw[dashed] (3.78,6) --(2.73,6.5); 
 
 
 \draw[very thick] (2.52,6) --(2.73,6.5); 
 
 
 \draw[very thick] (3.78,6) --(4.095,6.5); 
 
\draw (0,6.5) --(4.095,6.5)node[right] {$t_1$}; 
 
\draw[dashed] (4.095,6.5) --(2.9575,7.04167); 
 
 
 \draw[very thick] (2.73,6.5) --(2.9575,7.04167); 
 
 
 \draw[very thick] (4.095,6.5) --(4.43625,7.04167); 
 
\draw (0,7.04167) --(4.43625,7.04167)node[right] {$t_2$}; 
 
\draw[dashed] (4.43625,7.04167) --(3.20396,7.62847); 
 
 
 \draw[very thick] (2.9575,7.04167) --(3.20396,7.62847); 
 
 
 \draw[very thick] (4.43625,7.04167) --(4.80594,7.62847)node[sloped, below, midway] {$\hspace{10pt} ...$}; 
 
 
 \draw[very thick] (3.20396,7.62847) --(3.47096,8.26418); 
 
 
 \draw[very thick] (4.80594,7.62847) --(5.20643,8.26418); 
 
 
 \draw[very thick] (3.20396,7.62847) --(3.675,8.75)node[midway, right] {$r=\lambda' t$}; 
 
 
 \draw[very thick] (4.80594,7.62847) --(5.5125,8.75)node[midway, right] {$r=\lambda'' t$}; 
 
\draw (0,6) --(3.78,6)node[right] {$t_{0}$}; 
 
\draw[dashed] (3.78,6) --(2.205,5.25); 
 
 
 \draw[very thick] (2.52,6) --(2.205,5.25); 
 
 
 \draw[very thick] (3.78,6) --(3.3075,5.25); 
 
\draw (0,5.25) --(3.3075,5.25)node[right] {$t_{-1}$}; 
 
\draw[dashed] (3.3075,5.25) --(1.92937,4.59375); 
 
 
 \draw[very thick] (2.205,5.25) --(1.92937,4.59375); 
 
 
 \draw[very thick] (3.3075,5.25) --(2.89406,4.59375); 
 
\draw (0,4.59375) --(2.89406,4.59375)node[right] {$t_{-2}$}; 
 
\draw[dashed] (2.89406,4.59375) --(1.6882,4.01953); 
 
 
 \draw[very thick] (1.92937,4.59375) --(1.6882,4.01953); 
 

 \draw[very thick] (2.89406,4.59375) --(2.5323,4.01953)node[sloped, below, midway] {$...$}; 
 
 
 \draw[very thick] (1.6882,4.01953) --(1.47718,3.51709); 
 
 
 \draw[very thick] (2.5323,4.01953) --(2.21577,3.51709); 
 
 
 \draw[very thick] (1.47718,3.51709) --(1.29253,3.07745); 
 
 
 \draw[very thick] (2.21577,3.51709) --(1.9388,3.07745); 
 
 
 \draw[very thick] (1.47718,3.51709) --(1.2075,2.875); 
 
 
 \draw[very thick] (2.21577,3.51709) --(1.81125,2.875); 
 
\end{tikzpicture}
\endpgfgraphicnamed
\caption{Construction for the sequence $\{t_l\}$. Dashed lines either have slope $1$ or $-1$.} \label{fig1}
\end{center}
\end{figure}

Suppose for some $\tau_n \in [\frac{5}{8}T_n,\frac{7}{8}T_n]$, $E(U(x,\tau_n),B_{\lambda'\tau_n}(0)) \leq \alpha \eta$.
We will show that
\begin{align}
 & \En  = 
\notag \\
& \quad \frac{1}{T_n}\int_{\frac{5}{8}T_n}^{\tau_n}E(U(x,t),B_{\lambda't}(0))\,dt
+\frac{1}{T_n}\int_{\tau_n}^{\frac{7}{8}T_n}E(U(x,t),B_{\lambda't}(0))\,dt
< (1-\gamma) \eta,
\label{energy-split}
\end{align}
which would contradict our assumption \eqref{contradict-me}.
We seperately estimate the two integrals in the middle of \eqref{energy-split}.

Let $\{t_l\}_{l \in \mathbb Z}$ be defined by $t_0 = \tau_n$, $t_l = (\frac{1+\lambda''}{1+\lambda'})^lt_0$ for $l>0$,
and  $t_l = (\frac{1-\lambda''}{1-\lambda'})^{-l}t_0$ for $l < 0$ (see Figure 1).
Let $N_+=N_+(\lambda',\lambda'')$ be the smallest number with $t_{N_+} \ge \frac{7}{8}T_n$, and $N_-=N_-(\lambda',\lambda'')$ the smallest number with $t_{-N_-} \leq \frac{5}{8}T_n$.
Let $q_+=\frac{1+\lambda''}{1+\lambda'}$ and $q_-=\frac{1-\lambda''}{1-\lambda'}$.

We first estimate the integral over $[\tau_n,\frac{7}{8}T_n]$. 
By \eqref{nrgy-cnsrv},
\begin{align}
\sup_{t_{l} \leq t \leq t_{l+1}}E(U(x,t),B_{\lambda't}(0)) 
& \leq E(U(x,t_{l}),B_{\lambda't_{l}}(0))
\notag \\
& \quad +E(U(x,t_{l}),B_{\lambda''t_{l}}(0)\setminus B_{\lambda't_{l}}(0)).
\notag
\end{align}
From this, our bounds on $E(U(x,t_{l}),B_{\lambda''t_{l}}(0)\setminus B_{\lambda't_{l}}(0))$, and our assumption
on $t_0=\tau_n$,
\begin{align}
\sup_{t_{l} \leq t \leq t_{l+1}}E(U(x,t),B_{\lambda't}(0)) 
& \leq 
E(U(x,t_{0}),B_{\lambda't_{0}}(0))
\notag \\
& \quad + 
l\sup_{t_0 \leq t \leq t_l}E(U(x,t),B_{\lambda''t}(0)\setminus B_{\lambda't}(0))
\notag \\
& \leq
\alpha \eta +l \beta \eta.
\notag
\end{align}
Integrating over $[t_l,t_{l+1}]$, 
\begin{equation}
\int_{t_l}^{t_{l+1}}E(U(x,s), B_{\lambda's}(0))\,ds 
\leq 
(q_+-1)q_+^l\tau_n\alpha \eta +l(q_+-1)q_+^l\tau_n\beta \eta.
\label{sm-m}
\end{equation}
With \eqref{sm-m} and an elementary summation formula,
\begin{align}
& \int_{\tau_n}^{\frac{7}{8}T_n}E(U(x,s), B_{\lambda's}(0))\,ds 
 \leq
\sum_{l=0}^{N_+-1}
\int_{t_l}^{t_{l+1}}E(U(x,s), B_{\lambda's}(0))\,ds
\notag \\
& \quad =  
(q_+^{N_+}-1)\tau_n\alpha \eta 
\notag \\
& \quad \quad +
\left [(N_{+}+\frac{q_+}{1-q_+})q_+^{N_+}-\frac{q_+}{1-q_+} \right ]\tau_n\beta \eta
\notag \\
& \quad \leq
(q_+^{N_+}-1)\frac{7}{8}T_n \alpha \eta 
\notag \\
& \quad  \quad +
\left [(N_{+}-\frac{q_+}{1-q_+})q_+^{N_+}+\frac{q_+}{1-q_+} \right ]\frac{7}{8}T_n\beta \eta.
\notag
\end{align}
For the integral over $[\frac{5}{8}T_n,\tau_n]$ in \eqref{energy-split}, we use a similar
argument,
\begin{align}
& \int_{\frac{5}{8}T_n}^{\tau_n}E(U(x,s), B_{\lambda's}(0))\,ds 
 \leq 
\sum_{l=0}^{-N_-+1}
((q_--1)q_-^{-l}\tau_n(\alpha \eta -l(q_--1)q_-^{-l}\tau_n\beta \eta))
\notag \\
& \quad =  
(q_+^{N_-}-1)\tau_n\alpha \eta 
\notag \\
& \quad \quad +
\left [(N_{-}+\frac{q_-}{1-q_-})q_-^{N_-}-\frac{q_-}{1-q_-} \right ]\tau_n \beta \eta
\notag \\
& \quad \leq
(q_-^{N_-}-1)\frac{7}{8}T_n\alpha \eta 
\notag \\
& \quad \quad +
\left [(N_{-}-\frac{q_-}{1-q_-})q_-^{N_-}+\frac{q_-}{1-q_-}
\right ]\frac{7}{8}T_n\beta \eta. 
\notag
\end{align}
Combining these,
\begin{align}
 \frac{1}{T_n}\int_{\frac{5}{8}T_n}^{\tau_n}E(U(x,s), B_{\lambda's}(0))\,ds
& +
\frac{1}{T_n}\int_{\tau_n}^{\frac{7}{8}T_n}E(U(x,s), B_{\lambda's}(0))\,ds
  \leq
\frac{7}{8}(q_-^{N_-}+q_+^{N_+}-2)\alpha \eta
\notag \\
& \quad +\left [(N_{+}-\frac{q_+}{1-q_+})q_+^{N_+}+\frac{q_+}{1-q_+} \right ]\frac{7}{8}\beta \eta 
\notag \\
& \quad +
\left [(N_{-}-\frac{q_-}{1-q_-})q_-^{N_-}+\frac{q_-}{1-q_-} \right ]\frac{7}{8}\beta \eta.
\notag
\end{align}
Choosing $\alpha$, $\beta$, and $\gamma$ appropriately, we have that
\begin{align}
& \frac{7}{8}(q_-^{N_-}+q_+^{N_+}-2)\alpha \eta
+\left [(N_{+}-\frac{q_+}{1-q_+})q_+^{N_+}+\frac{q_+}{1-q_+} \right ]\frac{7}{8}\beta \eta 
\notag \\
& \quad +
\left [(N_{-}-\frac{q_-}{1-q_-})q_-^{N_-}+\frac{q_-}{1-q_-} \right ]\frac{7}{8}\beta \eta
 < (1-\gamma) \eta,
\notag
\end{align}
from which \eqref{energy-split} follows,
which contradicts \eqref{contradict-me}.
\end{proof}

With this lemma, along with Theorem \ref{c-t-z} and Theorem \ref{s}, we now prove Theorem \ref{main}.
\begin{proof}[of Theorem \ref{main}]
To begin with, we reproduce with only slight modification the argument of Struwe in \cite{MR1990477} on page 819. Let $\mathcal{T}=\cup_{n}[\frac{5}{8}T_n,\frac{7}{8}T_n]$, $\lambda'$ be as in \eqref{dfn-lmb}, and pick 
$R(t)$ so that for some sufficiently small $\eta_0$,
\begin{equation}
\eta_0 < E(U(x,t), B_{6 R(t)}(0)) < 2\eta_0 < \inf_{t \in \mathcal{T}}E(B_{\lambda' t}(0),t),
\label{nrg-trp}
\end{equation}
for $t \in \mathcal{T}$. That it is possible to pick such an $R(t)$ for small enough $\eta_0$ follows from Lemma \ref{lwr-nrg-bnd}.
With \eqref{nrg-trp} it can be shown that for $|\tau| < 5R(t)$,
\begin{equation}
\eta_0  \leq E(U(x,t+\tau), B_{11 R(t)}(0)), 
\label{nrgy-trp-2}
\end{equation}
and
\begin{equation}
 E(U(x,t+\tau),B_{R(t)}(0)) \leq 2\eta_0.
\label{sml-nrgy-trp-2}
\end{equation}

Since the intervals $\Lambda_l \equiv (t-R(t),t+R(t))$ cover $\mathcal{T}$, by Vitali's theorem we may select a countable, disjoint family $\{(t_l-R(t_l),t_l+R(t_l))\}_{l \in \mathbb N}=\{\Lambda_l\}_{l \in \mathbb N}$ such that
\begin{equation}
\mathcal{T} \subset \bigcup_{l \in \mathbb N}(t_l-5R(t_l),t_l+5R(t_l)).
\notag
\end{equation}
Let $R(t_l)=R_l$ and $\{(t_l-5R_l,t_l+5R_l)\}_{l \in \mathbb N}=\{\Lambda_l^*\}_{l \in \mathbb N}$. By possibly taking a subsequence and reordering, we may further assume that $t_l \rightarrow \infty$ and $t_l < t_{l+1}$.
Since $\lim_{t \rightarrow \infty}E(U(x,t), B_{\lambda''t}(0)\setminus B_{\lambda't}(0))=0$ for all $0 < \lambda' < \lambda'' < 1$,
 we have 
\begin{equation}
\lim_{l \rightarrow \infty} \frac{R_l}{t_l}=0. 
\label{r-dc-f}
\end{equation}

In order to show that there is a subsequence $\{t_{l_m}\}$ of $\{t_{l}\}$ with
\begin{equation}
\lim_{m \rightarrow \infty} \frac{1}{R_{l_m}}\int_{\Lambda_{l_m}}\int_{B_t(0)}|U_t|^2\,dx\,dt = 0,
\label{knt-ngr-dcy-sqn}
\end{equation}
we'll assume to the contrary that there is a $\delta>0$ and $l_0$
such that
\begin{equation}
\int_{\Lambda_l}\int_{B_t(0)}|U_t|^2\,dx\,dt \geq \delta R_l,
\label{knt-ngr-dcy-sqn-cnt}
\end{equation}
for $l\ge l_0$. For large enough $n$ so that $\sup \bigcup_{l < l_0} \Lambda_l^*<\frac{5}{8}T_n$,
let
\begin{align}
l_1 & = \max_k \left \{ k  \vert \hspace{5pt} [\frac{5}{8}T_n,\frac{7}{8}T_n] \subset \bigcup_{l \geq k} \Lambda_l^* \right \} \ge l_0,
\notag \\
\textrm{and }l_2 & = \min_k \left \{ k  \vert \hspace{5pt} [\frac{5}{8}T_n,\frac{7}{8}T_n] \subset \bigcup_{l_1 \leq k} \Lambda_l^* \right \}.
\notag
\end{align}
By taking $l_0$ large enough, we may assume from \eqref{r-dc-f} that
$R_l < \frac{1}{35} t_l$
for $l \geq l_0$.

From the maximality of $l_1$, minimality of $l_2$, and our assumptions on $R_l$, 
\begin{align}
\frac{5}{8}T_n & \leq t_{l_1}+5R_{l_1} \leq \frac{8}{7}t_{l_1},
\notag \\
\textrm{and }\frac{6}{7}t_{l_2} & \leq t_{l_2}-5R_{l_2} \leq \frac{7}{8}T_n.
\label{m-rds-bnd}
\end{align}
Because $\{t_l\}$ is increasing, and \eqref{m-rds-bnd},
we infer that $t_l-5R_l \geq \frac{77}{192}T_n$ and $ \frac{21}{20}T_n \geq t_l +5 R_l$ when $l_1 \leq l \leq l_2$.
It then follows that
\begin{equation}
\bigcup_{l_1 \leq l \leq l_2} \Lambda_l \subset [\frac{77}{192}T_n,\frac{21}{20}T_n].
\notag
\end{equation}
With this, the fact that $[\frac{5}{8}T_n,\frac{7}{8}T_n] \subset \bigcup_{l_1 \leq l \leq l_2} \Lambda_l^*$, and \eqref{knt-ngr-dcy-sqn-cnt}, we have that
\begin{align}
 \frac{\delta}{4}T_n & \leq \delta \sum_{l_1 \leq l \leq l_2}\textrm{diam}\Lambda_l^*  =
10\delta \sum_{l_1 \leq l \leq l_2}R_l 
\leq 10 \sum_{l_1 \leq l \leq l_2}
\int_{\Lambda_l}\int_{B_t(0)}|U_t|^2\,dx\,dt
\notag \\
& =
 10 
\int_{\bigcup_{l_1 \leq l \leq l_2}\Lambda_l}\int_{B_t(0)}|U_t|^2\,dx\,dt
\leq
10 
\int_{K_{\frac{77}{192}T_n}^{\frac{21}{20}T_n}}\int_{B_t(0)}|U_t|^2\,dx\,dt.
\notag
\end{align}
For big enough $T_n$, this contradicts \eqref{kntc-nrgy-dc}, thereby proving \eqref{knt-ngr-dcy-sqn}. For notational convenience, we refer to the subsequence satisfying \eqref{knt-ngr-dcy-sqn} as $\{t_l\}$.

Rescale with 
$U_l(t,x)=U(t_l+R_l t,R_l x)$ so that
\begin{equation}
\int_{-1}^{1}\int_{D_l(t)} |\partial_t U_l|^2 \,dx \,dt \rightarrow 0,
\label{ke-vanish}
\end{equation}
with
\begin{equation}
D_l(t)=\{x \vert \hspace{5pt} R_l|x| \leq \lambda' (t_l+R_l t)\}.
\notag
\end{equation}

From \eqref{sml-nrgy-trp-2}, \eqref{ke-vanish}, and \eqref{r-dc-f}, Theorem \ref{s} applies, so that locally the energy of $U_l$ decays as $l \rightarrow \infty$, which contradicts 
\eqref{nrgy-trp-2}.
 Therefore 
\begin{equation}
\lim_{T \rightarrow \infty}\E = 0.
\label{nrg-vnshs}
\end{equation}
With \eqref{nrg-vnshs}, we can now apply Theorem \ref{c-t-z} to show that $U$ scatters. From an integral formulation of \eqref{wv-mp-q},
\begin{align}
U(x,t)& =
\cos(t \sqrt{-\Delta})U_{0}(x) + \frac{\sin(t\sqrt{-\Delta})}{\sqrt{-\Delta}}U_{1}(x) \notag \\
& \quad -\int_{0}^{t}\frac{\sin((t-\tau)\sqrt{-\Delta})}{\sqrt{-\Delta}}B(U)(\partial_{\alpha}U,\partial^{\alpha}U)\,d\tau,
\notag
\end{align}
it is easy to see that when $U$ scatters,
\begin{align}
U_{+,0}(x) & =U_{0}(x)+\int_{0}^{\infty}\frac{\sin(\tau\sqrt{-\Delta})}{\sqrt{-\Delta}}B(U)(\partial_{\alpha}U,\partial^{\alpha}U)\,d\tau,
\notag \\
\textrm{and} 
\notag \\
U_{+,1}(x) & =U_{1}(x)-\int_{0}^{\infty}\cos(\tau\sqrt{-\Delta})B(U)(\partial_{\alpha}U,\partial^{\alpha}U)\,d\tau.
\notag
\end{align}

Therefore to prove scattering, it will suffice to show by energy estimates that 
\begin{equation}
\left \| \int_{0}^{\infty}\frac{\sin(\tau\sqrt{-\Delta})}{\sqrt{-\Delta}}B(U)(\partial_{\alpha}U,\partial^{\alpha}U)\,d\tau \right \|_{\dot{H}^1}
 \lesssim
\| |B(U)(\partial_{\alpha}U,\partial^{\alpha}U)| \|_{L_t^{1}L_x^2} 
\notag  
\end{equation}
is finite. 
Since $N$ is compact and $B$ is bilinear and symmetric, we may assume that there is a $b \in \mathbb R^{+}$ so that
\begin{equation}
|B(U)(\partial_{\alpha}U,\partial^{\alpha}U)|= 2|B(U)(\partial_{u}U,\partial_{v}U)|
\leq 
b|\partial_{u} U||\partial_{v}U|.
\notag
\end{equation}
Using Theorem \ref{c-t-z}, and the fact that for positive $r$ and $t$, 
\begin{equation}
\frac{1}{(r+t)^2+1}  \leq \min \left \{\frac{1}{r^2+1},\frac{1}{t^2+1} \right \},
\notag
\end{equation}
we have that
\begin{align}
& \| |B(U)(\partial_{\alpha}U,\partial^{\alpha}U)| \|_{L_t^{1}L_x^2} 
 \leq
b\int_{0}^{\infty}\left (
\int_{0}^{\infty}
|\partial_uU|^2 |\partial_vU|^2r\,dr
\right )^{\frac{1}{2}}\,dt
\notag \\
& \quad \leq
bc\int_{0}^{\infty}\left (
\int_{0}^{\infty}\frac{r}{((r+t)^2+1)^2((r-t)^2+1)}\,dr
\right )^{\frac{1}{2}}\,dt
\notag \\
& \quad \leq
bc\int_{0}^{\infty}(t^2+1)^{-3/4}\left (
\int_{0}^{\infty}\frac{r}{\sqrt{r^2+1}((r-t)^2+1)}\,dr
\right )^{\frac{1}{2}}\,dt
< \infty.
\label{sct}
\end{align}
By \eqref{sct}, it follows that $U$ scatters.

\end{proof}

\noindent
\textbf{Acknowledgments:} I would like to thank Joachim Krieger and
Sohrab Shahshahani for helpful discussions, and the referee for valuable comments.

\end{document}